\documentclass[10pt]{amsart}
\usepackage{amsmath}
\usepackage{amscd}
\usepackage[all]{xy}

\newcommand{\Pic}{\operatorname{Pic}}
\newcommand{\Cl}{\operatorname{Cl}}
\newcommand{\Div}{\operatorname{Div}}
\renewcommand{\div}{\operatorname{div}}
\newcommand{\red}{{\operatorname{red}}}
\newcommand{\nor}{{\operatorname{nor}}}
\newcommand{\sn}{{\operatorname{sn}}}

\newcommand{\Hom}{\operatorname{Hom}}   

\newcommand{\Cart}{\operatorname{Cart}}

\newcommand{\MSpec}{\operatorname{MSpec}}
\newcommand{\MProj}{\operatorname{MProj}}
\newcommand{\Ann}{\mathrm{ann}}
\newcommand{\Ass}{\mathrm{Ass}}
\newcommand{\nil}{\mathrm{nil}}

\def\lra{\longrightarrow}
\def\map#1{{\buildrel #1 \over \lra}}

\renewcommand{\P}{\mathbb{P}}

\newcommand{\bA}{\mathbb{A}}

\newcommand{\Z}{\mathbb{Z}}
\newcommand{\frakm}{\mathfrak{m}}
\newcommand{\frakp}{\mathfrak{p}}
\newcommand{\frakq}{\mathfrak{q}}
\newcommand{\cA}{\mathcal{A}}
\newcommand{\cF}{\mathcal{F}}
\newcommand{\cH}{\mathcal{H}}

\newcommand{\cL}{\mathcal{L}}

\numberwithin{equation}{section}

\theoremstyle{plain}
\newtheorem{thm}[equation]{Theorem}
\newtheorem{prop}[equation]{Proposition}
\newtheorem{lem}[equation]{Lemma}
\newtheorem{cor}[equation]{Corollary}

\theoremstyle{definition}
\newtheorem{defn}[equation]{Definition}
\newtheorem{exam}[equation]{Example}
\newtheorem{rem}[equation]{Remark}

\newtheorem{substuff}{Remark}[equation]
\newtheorem{subex}[substuff]{Example}
\newtheorem{subrem}[substuff]{Remark}

\begin{document}
 
\title[\ ] 
{Picard groups and Class groups of monoid schemes}

\author{J. Flores and C. Weibel}
\address{Department of Mathematics, Rutgers University, 
New Brunswick, NJ U.S.A.}
\email{jarflores@gmail.com, weibel@math.rutgers.edu}
\date{\today}
\bigskip

\begin{abstract} 
We study the Picard group of a monoid scheme and the 
class group of a normal monoid scheme. To do so, we develop
some ideal theory for (pointed abelian) noetherian monoids, including
primary decomposition and discrete valuations. The normalization of 
a monoid turns out to be a monoid scheme, but not always a monoid.
\end{abstract}
\maketitle

\section*{Introduction}

The main purpose of this paper is to study the Picard group of a monoid
scheme $X$ and, if $X$ is normal, the Weil divisor class group of $X$.
Along the way, we develop the analogues for pointed abelian monoids
of several notions in commutative ring theory such as 
associated prime ideals, discrete valuations and normalization.

Recall from \cite{chww} or \cite{Vezz} that a {\it monoid scheme} is a 
topological space $X$ equipped with a sheaf $\cA$ of pointed abelian monoids
such that locally $X$ is $\MSpec(A)$, the set of prime ideals of a
pointed monoid $A$ with its natural structure sheaf. The interest
in monoid schemes stems from the fact that every toric variety 
(or rather its fan $\Delta$) is associated to a monoid scheme 
$X_\Delta$ in a natural way; see \cite[4.2]{chww}. 

To study the divisor class group of a normal monoid, we need to
establish the connection between height one primes and discrete 
valuations of a normal monoid. To this end, the first few sections
develop the theory of associated primes for pointed abelian 
noetherian monoids, 
in analogy with the theory for commutative rings \cite{AM}.

The notion of normalization behaves differently for monoids than
it does in ring theory.  The normalization of a cancellative monoid
is well known \cite{Gilmer}, but the surprise is that we can make sense 
of normalization for partially cancellative monoids,
i.e., for arbitrary quotients of cancellative monoids. The normalization
of $A$ turns out to be not a monoid but a {\it monoid scheme}: the disjoint
union over the minimal primes of the normalizations of the $A/\frakp$.
That is, we need to embed monoids (contravariantly) into the 
larger category of monoid schemes.

With these preliminaries, we define the Weil divisor class group of
a normal monoid or monoid scheme in the same way as in algebraic geometry,
and give its basic properties in Section \ref{sec:Weil}.
It turns out that the class group of a toric monoid scheme agrees with
the divisor class group of the associated toric variety.

The {\it Picard group} $\Pic(X)$ of a monoid scheme $X$ is the set
of isomorphism classes of invertible sheaves on $X$ 
(see Definition \ref{def:Pic}, \cite{CLS} or \cite{GHS}).
In more detail, recall that if $A$ is a pointed monoid then an {\it$A$-set}
is a pointed set $L$ on which $A$ acts; the smash product
$L_1\wedge_A L_2$ is again an $A$-set.
Similarly, if $(X,\cA)$ is a monoid scheme then a sheaf of $\cA$-sets
is a sheaf of pointed sets $\cL$, equipped with a pairing 
$\cA\wedge\cL\to\cL$ making each stalk $\cL_x$ an $\cA_x$-set.
We say that $\cL$ is {\it invertible} if it is locally isomorphic to 
$\cA$ in the Zariski topology; the smash product of invertible sheaves
is again invertible, and the dual sheaf $\cL^{-1}$ satisfies
$\cL\wedge_A\cL^{-1}\cong\cA$. This gives $\Pic(X)$ the structure
of an abelian group; see Section \ref{sec:Pic}.

The group of $A$-set automorphisms of any monoid $A$ is
canonically isomorphic to $A^\times$. Thus a standard argument 
(see Lemma \ref{Pic=H1}) shows that $\Pic(X) \cong H^1(X,\cA^\times)$.
Using this, we show in \ref{Pic.HI} and \ref{Pic.Xsn} that
(unlike algebraic geometry) 
we always have $\Pic(X)=\Pic(X\times\bA^1)$ and $\Pic(X)=\Pic(X_\sn)$,
where $X_\sn$ is the seminormalization of $X$.

In Section \ref{sec:Cartier} we show that $\Pic(X)$ is always a subgroup
of $\Cl(X)$ when $X$ is normal. It turns out that the Picard group of 
a toric monoid scheme agrees with the Picard group of the associated 
toric variety; see Theorem \ref{thm:toric}. This explains the
calculation by Vezzani (our original inspiration) that the Picard group 
of the projective monoid scheme $\P^n$ is $\Z$. We conclude this paper
by providing some exact sequences relating $\Pic(X)$ to $\Pic(X_\nor)$.


After this paper was written, we were informed by Bernhard K\"ock
that many of the results in the first two sections were obtained
by Ulrich Kobsa in \cite{Kobsa}, and by Franz Halter-Koch in \cite{H-K}.
Since these sources are couched in a different framework, we have
elected to retain our sections in order to be self-contained.

\medskip
\section*{\it Notation}
In this paper the term `monoid' will always mean a pointed, abelian
monoid unless otherwise stated. We write the product multiplicatively,
so that the zero element $0$ is the basepoint and $1$ is the unit. We also
write $A^\times$ for the group of invertible elements in $A$.


We remind the reader of some basic facts concerning such monoids.  Let
$A$ be a monoid.  An ideal $I\subseteq A$ is a subset of elements such
that $AI\subseteq I$; a nonzero ideal of the form $Ax$ is called a 
{\it principal ideal}.  The quotient or factor monoid
$A/I$ identifies all elements of the ideal $I$ with 0.  If $B$ 
is an unpointed monoid, we can form the (pointed) monoid $B_+$ by adding
a disjoint zero element.

An ideal $\frakp\subseteq A$ is {\it prime} when $xy\in\frakp$ means
$x\in\frakp$ or $y\in\frakp$.  The set of all prime ideals form a
topological space, denoted $\MSpec(A)$, under the Zariski Topology. 
If the elements of $A$ satisfy $ab=ac$ implies $b=c$,
we say that $A$ is {\it cancellative}. If $A$ is the quotient of a
cancellative monoid by an ideal, we say it is {\it partially cancellative},
or {\it pc}.

Let $S\subseteq A$ be a multiplicatively closed subset.  We can then
form the monoid $S^{-1}A$, termed the localization of $A$ at $S$,
which is obtained from $A$ by inverting all elements of $S$.  When
$S=A\backslash\frakp$, we say that $S^{-1}A$ is obtained from $A$ by
localizing at $\frakp$ and denote this special case by $A_{\frakp}$.
Note that $\frakp$ is the maximal ideal of $A_{\frakp}$.  

In the case where $\frakp=\{0\}$ is a prime ideal, we write $A_0$ 
for $A_\frakp$; it is an abelian group with a disjoint basepoint,
called the group completion of $A$ and there is an inclusion 
$A\hookrightarrow A_0$ precisely when $A$ is cancellative.
We shall write $A_0^\times$ for the abelian group $A_0\setminus\{0\}$;
it is the classical group completion of the unpointed monoid 
$A\setminus\{0\}$.

\medskip
\paragraph{\it Acknowledgements}
The authors are grateful to Alberto Vezzani and Luca Barbieri
for asking about the structure of $\Pic(X)$, and to Bernhard K\"ock
for pointing out the references \cite{Kobsa} and \cite{H-K}.

\newpage
\section{Primary decomposition}

The definitions and proofs in this section are direct translations
from commutative ring theory and are given here for convenience.

We say that a monoid $A$ is {\it noetherian} when it satisfies
the ascending chain condition on ideals.  This is equivalent to
the condition that every ideal is finitely generated, by the usual
proof: if $I_1\subseteq I_2\subseteq\cdots$ is an ascending chain 
of finitely generated ideals of $A$, then the union is finitely 
generated if and only if it equals some $I_n$.

%
%

\begin{rem}
When $A^{\times}$ is a finitely generated abelian group, a monoid $A$ 
is noetherian if and only if $A$ satisfies the ACC on congruences (which need not
be ideals), i.e.,
the descending chain condition on \emph{all} quotient monoids. We will not
consider this (slightly) stronger chain condition.
\end{rem}

In order to establish a primary decomposition for ideals, 
we introduce some definitions.
A proper ideal $\frakq\subseteq A$ is {\it primary} when $xy\in\frakq$
implies $x\in\frakq$ or $y^n\in\frakq$.  Alternatively, $\frakq$
is primary when every zero-divisor $a\in A/\frakq$ is nilpotent.
The \emph{radical} of an ideal $I$ is $\sqrt{I}=\{a\in A\ |\ a^n\in I\}$;
it is a prime ideal when $I$ is primary.
An ideal $I\subseteq A$ is said to be {\it irreducible} when $I=J\cap K$, 
with $J,K\subseteq A$ ideals, implies $I=J$ or $I=K$.

\begin{lem}\label{irred.prim}
Every irreducible ideal of a noetherian monoid is primary.
\end{lem}

\begin{proof}
An ideal $I\subseteq A$ is primary if and only if the zero
ideal of $A/I$ is primary.  Therefore we need only show when $(0)$ is
irreducible, it is primary.  Let $xy=0$ with $y\neq 0$; we will
show $x^n=0$.  Consider the ascending chain $\Ann(x)\subseteq
\Ann(x^2)\subseteq\cdots$, where $\Ann(x)=\{a\in A\ | \ ax=0\}$.
By the noetherian property, this chain
must stabilize, say $\Ann(x^n)=\Ann(x^{n+1})=\cdots$ for some $n>0$.

We claim that $0=(x^n)\cap (y)$.  Let $a\in(x^n)\cap(y)$,
say $a=bx^n=cy$. Then $0=c(xy)=ax=(bx^n)x=bx^{n+1}$.
Hence $b\in\Ann(x^{n+1})=\Ann(x^n)$ giving $a=bx^n=0$.  Since $(0)$ is
irreducible and $y\neq 0$, we must have $x^n=0$ proving $(0)$ is
primary.
\end{proof}

%

\begin{thm}[Primary decomposition]\label{p.decomp}
In a noetherian monoid every ideal $I$ can be written as the finite
intersection of irreducible primary ideals $I=\cap_i\frakq_i$.
\end{thm}

\begin{proof}
Suppose the result is false.  Since $A$ is noetherian, the set of 
ideals which cannot be written as a finite intersection of irreducible 
ideals has a maximal element, say $I$.
Since $I$ is not irreducible, it can be written $I=J\cap K$ where 
$J,K$ are ideals of $A$ containing $I$.  By maximality, 
both $J$ and $K$ (and hence $I$) can be written as a finite intersection of 
irreducible ideals.  This is a contradiction, and the theorem
follows via Lemma \ref{irred.prim}.
\end{proof}

\begin{subrem}\label{nil(A)}
We say that $A$ is {\it reduced} if whenever $a,b\in A$ satisfy 
$a^2=b^2$ and $a^3=b^3$ then $a=b$. This implies that $A$ has no 
nilpotent elements, i.e., that
$\nil(A)=\{ a\in A: a^n=0 \text{ for some } n\}$ vanishes.
By Theorem \ref{p.decomp}, $\nil(A)$
is the intersection of all the prime ideals in $A$; 
cf.\ \cite[(1.1)]{chww-p}.

There is a universal map $A\to A_\red$ from $A$ to reduced monoids;
$A_\red$ is a quotient of $A/\nil(A)$. The universal non-reduced 
example $A_u=\langle a,b:a^n=b^n\textrm{ for }n\ge2\rangle$ shows
that being reduced is stronger than having $\nil(A)=0$. This notion
of reduced avoids pathologies (such as the embedded prime in $\Z[A_u]$),
and is important in \cite{chww-p}.
When $A$ is a pc monoid, $A$ is reduced if and only if $\nil(A)=0$; 
in this case $A_\red=A/\nil(A)$ is a reduced monoid (see \cite[1.6]{chww-p}).
\end{subrem}

Let $A$ be a noetherian monoid and $I\subseteq A$ an ideal.
Given a minimal primary decomposition of $I$, $I=\cap_i\frakq_i$,
$\Ass(I)$ denotes the set of prime ideals occurring as the radicals 
$\frakp_i=\sqrt{\frakq_i}$; the $\frakp_i$ are called the 
associated primes of $I$. Although the primary decomposition need 
not be unique, the set $\Ass(I)$ of 
associated primes of $I$ is independent of the minimal
primary decomposition, by \ref{ap.radquot} below.

In order to show that the associated primes are ideal quotients,
we recall the definition. Given ideals $I,J$ of $A$, the 
{\it ideal quotient} of $I$ by $J$ is the ideal 
$(I:J)=\{a\in A \ | \ ax\in I \text{ for all } x\in J \}$.  
When $I=aA$ is a principal ideal, we often write $(a:J)$ for $(aA:J)$.
The next two lemmas are simple, direct translations from ring
theory (such as \cite[1.11, 4.4]{AM}) so their proofs are omitted.

\begin{lem}\label{pcont.quotprim}
Let $\frakp$ be prime ideal in a monoid $A$.
\begin{enumerate}
\item[i)] If $I_1,\ldots,I_n$ are ideals such that 
$\cap_i I_i\subseteq\frakp$, then $I_i\subseteq\frakp$ for some $i$.  
If in addition $\frakp=\cap_i I_i$, then $\frakp=I_i$ for some $i$.
\item[ii)] Let $\frakq$ be a $\frakp$-primary ideal of $A$.
  If $a\in A\backslash\frakq$, then $(\frakq:a)$ is $\frakp$-primary.
\end{enumerate}
\end{lem}

\begin{lem}\label{rad.power}
Let $A$ be a noetherian monoid.
For any ideal $I$ of $A$,  $(\sqrt{I})^n\subseteq I$ for some $n$.
\end{lem}


\begin{prop}\label{ap.radquot}
Let $I\subseteq A$ be an ideal with minimal primary decomposition
$I=\cap_{i=1}^n\frakq_i$ where $\frakq_i$ is $\frakp_i$-primary.  Then $\Ass(I)$ is 
exactly the set of prime ideals which occur in the set of ideals 
$\sqrt{(I:a)}$, where $a\in A$. Hence $\Ass(I)$ is independent of the
choice of primary decomposition.

In addition, the minimal elements in $\Ass(I)$ are exactly the
set of prime ideals minimal over $I$.
\end{prop}

\begin{proof} (Compare \cite[4.5, 4.6]{AM}.)
First note that 
\[
\sqrt{(I:a)} = \sqrt{(\cap_i\frakq_i:a)} 
= \sqrt{\cap_i(\frakq_i:a)} = \cap_i \sqrt{(\frakq_i:a)}.
\]
Since $\frakq_i$ is $\frakp_i$-primary, this equals 
$\cap_{a\not\in\frakq_i}\frakp_i$, by Lemma~\ref{pcont.quotprim}(ii).
If $\sqrt{(I:a)}$ is prime, then it is $\frakp_i$ for some $i$, by
Lemma~\ref{pcont.quotprim}(i).  Conversely, by minimality of the primary
decomposition, for each $i$ there exists an $a_i\not\in\frakq_i$ but
$a_i\in\cap_{j\neq i}\frakq_j$.  Using Lemma~\ref{pcont.quotprim}(i)
once more, we see $\sqrt{(I:a_i)}=\frakp_i$.

Finally, if $I\subseteq\frakp$ then $\cap\frakp_i\subseteq\frakp$,
so $\frakp$ contains some $\frakp_i$ by Lemma \ref{pcont.quotprim}(i).
If $\frakp$ is minimal over $I$ then necessarily $\frakp=\frakp_i$.
\end{proof}

\begin{prop}\label{ap.ann}
Let $A$ be a noetherian monoid and $I\subseteq A$ an ideal.
Then the associated prime ideals of $I$
are exactly the prime ideals occurring in the set of ideals
$(I:a)$ where $a\in A$.
\end{prop}

\begin{proof}
The ideals $I_i=\cap_{j\neq i}\frakq_j$ strictly contain $I$ by
minimality of the decomposition.  Since $\frakq_i\cap I_i=I$, any
$a\in I_i\backslash I$ is not contained in $\frakq_i$, hence
$(I:a)$ is $\frakp_i$-primary by Lemma \ref{pcont.quotprim}(ii).
Now, by Lemma~\ref{rad.power} we have
$\frakp_i^n\subseteq\frakq_i$ for some $n>0$, hence
\[
\frakp_i^nI_i\subseteq \frakq_iI_i\subseteq \frakq_i\cap I_i=I.
\]
Choose $n$ minimal so that $\frakp_i^nI_i\subseteq I$ (hence in $I_i$)
and pick $a\in\frakp_i^{n-1}I_i$ with $a\not\in I$. Since 
$\frakp_ia\subseteq I$ we have $\frakp_i\subseteq (I:a)$; 
as $(I:a)$ is $\frakp_i$-primary, we have $\frakp_i=(I:a)$.  
Conversely, if $(I:a)$ is prime then it is an
associated prime by Proposition~\ref{ap.radquot}.
\end{proof}

\begin{lem}\label{local.zero}
Let $A$ be a noetherian monoid and $I\subseteq A$ an ideal.
\begin{enumerate}
\item[i)] If $I$ is maximal among ideals of the form $(0:a)$, $a\in A$, 
then $I$ is an associated prime of $0$.
\item[ii)] If $a\in A$, then $a=0$ in $A$ if and only if 
$a=0$ in $A_{\frakp}$ for every prime $\frakp$ associated to $0$.
\end{enumerate}
\end{lem}

\begin{proof}
Suppose that $I=(0:a)$ is maximal, as in (i). 
If $xy\in I$ but $y\not\in I$, then $axy=0$ and $ay\ne0$.  
Hence $I\subseteq I\cup Ax\subseteq (0:ay)$; by maximality, 
$I=I\cup Ax$ and $x\in I$. Thus $I$ is prime; by 
Proposition~\ref{ap.radquot}, $I$ is associated to $(0)$.

Suppose that $0\neq a\in A$, and set $I=(0:a)$.
By (i), $I\subseteq\frakp$ for some associated prime $\frakp$.
But then $a\ne0$ in $A_{\frakp}$.
\end{proof}

\section{Normal and Factorial monoids}


In this section, we establish the facts about normal monoids
needed for the theory of divisors.

The vocabulary for integral extensions of monoids mimicks that for
commutative rings. If $A$ is a submonoid of $B$, we say that an
element $b\in B$ is \emph{integral} over $A$ when $b^n\in A$ for 
some $n>0$, and the {\it integral closure} of $A$ in $B$ is the
submonoid of elements integral over $A$. If $A$ is a cancellative
monoid, we say that it is {\it normal} (or integrally closed) if
it equals its integral closure in its group completion.
(See  \cite[1.6]{chww}.)

\begin{exam}\label{ex:normal}
It is elementary that all factorial monoids are normal.
The affine toric monoids of \cite[4.1]{chww} are normal, and so are
arbitrary submonoids of a free abelian group closed under divisibility.
By \cite[4.5]{chww},
every finitely generated normal monoid is $A\wedge U_*$ for
an affine toric monoid $A$ and a finite abelian group $U$.
\end{exam}

We now present some basic facts concerning normal monoids which
parallel results for commutative rings.

\begin{lem}\label{wedge.normal}
If $A_1$ and $A_2$ are normal monoids, so is $A_1\wedge A_2$
\end{lem}

\begin{proof}
The group completion of $A_1\wedge A_2$ is
$(A_1\wedge A_2)^+=A_1^+\wedge A_2^+$. If 
$a_1\wedge a_2\in A_1^+\wedge A_2^+$ is integral over $A_1\wedge A_2$,
then $(a_1\wedge a_2)^n=a_1^n\wedge a_2^n\in A_1\wedge A_2$, hence
$a_1^n\in A_1$, $a_2^n\in A_2$.  By normality, $a_1\in A_1$ and
$a_2\in A_2$.
\end{proof}

\begin{lem}\label{local.normal}
Let $A\subseteq B$ be monoids and $S\subseteq A$ be multiplicatively closed. 
We have the following:
\begin{enumerate}
\item[i)] If $B$ is integral over $A$, 
then $S^{-1}B$ is integral over $S^{-1}A$.
\item[ii)]  If $B$ is the integral closure of $A$ in a monoid $C$, 
then $S^{-1}B$ is the integral closure of $S^{-1}A$ in $S^{-1}C$.
\item[iii)]  If $A$ is normal, then $S^{-1}A$ is normal.
\end{enumerate}
\end{lem}

\begin{proof}
Suppose that $b$ is integral over $A$, i.e., $b^n\in A$ for some $n>0$.
Then $b/s\in S^{-1}B$ is integral over $S^{-1}A$ because
$(b/s)^n\in S^{-1}A$. This proves (i). For (ii), it suffices by (i) to
suppose that $c/1\in S^{-1}C$ is integral over $S^{-1}A$ and show
that $c/1$ is in $S^{-1}B$. If $(c/1)^n=a/s$ in $S^{-1}A$ then 
$c^nst=at$ in $A$ for some $t\in S$. Thus $cst$ is in $B$, 
and $c/1=(cst)/st$ is in $S^{-1}B$. 
It is immediate that ii) implies iii).
\end{proof}

\begin{lem}\label{cayley.int}
Let $A$ be a cancellative monoid and $I\subseteq A$ a finitely
generated ideal.  If $u\in A_0$ is such that $u I\subseteq I$, then
$u$ is integral over $A$.
\end{lem}


\begin{proof}
Let $X=\{x_1,\ldots,x_r\}$ be the set of generators of $I$.
Since $uI\subseteq I$, there is a function $\phi:X\to X$ such that
$ux\in A\phi(x)$ for each $x\in X$.
Since $X$ is finite there is an $x\in X$ and an $n$ so that 
$\phi^n(x)=x$. For this $x$ and $n$ there is an
$a\in A$ so that $u^nx=ax$. By cancellation, $u^n=a$.
\end{proof}

\smallskip
\paragraph{\it Discrete valuation monoids}
Recall from \cite[8.1]{chww} that a \emph{valuation monoid} 
is a cancellative monoid $A$ such that for every non-zero 
$\alpha$ in the group completion $A_0$, either $\alpha\in A$ or 
$\alpha^{-1}\in A$. Passing to units, we see that $A^{\times}$ 
is a subgroup of the abelian group $A_0^\times$, and the 
{\it value group} is the quotient $A_0^\times/A^\times$.
The value group is a totally ordered abelian group 
($x\geq y$ if and only if $x/y\in A$). 
Following \cite[8.3]{chww}, we call $A$ a {\em discrete valuation
monoid}, or DV monoid for short, if the value group is infinite cyclic.
In this case, a lifting $\pi\in A$ of the positive generator of 
the value group generates the maximal ideal $\frakm$ of $A$ and
every $a\in A$ can be written $a=u\pi^n$ for some 
$u\in A^{\times}$ and $n\ge0$.
Here $\pi$ is called a \emph{uniformizing parameter} for $A$.

It is easy to see that valuation monoids are normal, and that
noetherian valuation monoids are discrete \cite[8.3.1]{chww}. We now
show that one-dimensional, noetherian normal monoids are DV monoids.

\begin{prop}\label{dvm.normal}
Every one-dimensional, noetherian normal monoid 
is a discrete valuation monoid (and conversely). 
\end{prop}

\begin{proof}
%
Suppose $A$ is a one-dimensional noetherian normal monoid, and 
choose a nonzero $x$ in the maximal ideal $\frakm$.  
By primary decomposition \ref{p.decomp}, $\sqrt{xA}$ must be $\frakm$ 
and (by Lemma \ref{rad.power}) there is an $n>0$ with
$\frakm^n\subseteq xA$, $\frakm^{n-1}\not\subseteq xA$.
Choose $y\in \frakm^{n-1}$ with $y\not\in xA$ and set $\pi=x/y\in A_0$.
Since $\pi^{-1}\not\in A$ and $A$ is normal, $\pi^{-1}$ is not 
integral over $A$.  By Lemma~\ref{cayley.int},
$\pi^{-1}\frakm\not\subseteq \frakm$; since $\pi^{-1}\frakm\subseteq A$ by
construction, we have $\pi^{-1}\frakm=A$, or $\frakm=\pi A$.

Lemma~\ref{cayley.int} also implies that $\pi^{-1}I\not\subseteq I$
for every ideal $I$. If $I\ne A$ then $I\subseteq\pi A$ so
$\pi^{-1}I\subseteq A$. Since $I=\pi^{-1}(\pi I)\subset\pi^{-1}I$,
we have an ascending chain of ideals which must terminate at
$\pi^{-n}I=A$ for some $n$. Taking $I=aA$, this shows that
every element $a\in A$ can be written $u\pi^n$ for a unique $n\ge0$
and $u\in A^\times$. Hence
every element of $A_0$ can be written $u\pi^n$ for a unique $n\in\Z$
and $u\in A^\times$, and the valuation map
$\mathrm{ord}:A_0\rightarrow \Z\cup\{\infty\}$ defined by
$\mathrm{ord}(u\pi^n)=n$ makes $A$ a discrete valuation monoid.
\end{proof}

\begin{cor}\label{DVmonoid}
If $\frakp$ is a height one prime ideal of a noetherian normal monoid $A$ 
then $A_{\frakp}$ is a discrete valuation monoid (DV monoid).
\end{cor}

\begin{proof}
The monoid $A_{\frakp}$ is one dimensional and normal
by Lemma~\ref{local.normal}.  Now use Proposition~\ref{dvm.normal}.
\end{proof}


\begin{lem}\label{prime.principal}
If $A$ is a noetherian normal monoid, and $\frakp$ is a prime ideal
associated to a principal ideal, then $\frakp$ has height one and
$\frakp_{\frakp}$ is a principal ideal of $A_{\frakp}$.
\end{lem}

\begin{proof}
Let $a\in A$ and $\frakp$ a prime ideal associated to $aA$ so that by
Proposition \ref{ap.ann}, $\frakp=(a:b)$ for some $b\in A\backslash
aA$.  To show $\frakp_{\frakp}\subseteq A_{\frakp}$ is principal, we
may first localize and assume that $A$ has maximal ideal $\frakp$.
Let $\frakp^{-1}=\{u\in A_0 \ | \ u\frakp\subseteq A\}$.  Since
$A\subseteq \frakp^{-1}$, we have $\frakp\subseteq
\frakp^{-1}\frakp\subseteq A$, and since $\frakp$ is maximal, we must
have $\frakp^{-1}\frakp=\frakp$ or $\frakp^{-1}\frakp=A$.

If $\frakp^{-1}\frakp=\frakp$, every element of $\frakp^{-1}$ must be
integral over $A$ by Lemma~\ref{cayley.int}.  Since $A$ is normal,
$\frakp^{-1}\subseteq A$, hence $\frakp^{-1}=A$ and $\frakp b\subseteq
aA$ implies $b/a\in\frakp^{-1}=A$.  This is only the case if
$b\in aA$, since $a$ is not a unit, contradicting the assumption.
Therefore $\frakp^{-1}\frakp=A$ and there exists $u\in\frakp^{-1}$
with $u\frakp=A$, namely $\frakp=u^{-1}A$.
\end{proof}
To finish the section we show that any noetherian normal monoid 
is the intersection of its localizations at height one primes.  As
with Corollary \ref{DVmonoid}, this result parallels the situation in
commutative rings.

\begin{thm}\label{intersect.Ap}
A noetherian normal monoid $A$ is the intersection of the 
$A_{\frakp}$ as $\frakp$ runs over all height one primes of $A$.
\end{thm}

\begin{proof}
That $A$ is contained in the intersection is clear.  Now, suppose
$a/b\in A_0\backslash A$ so that $a\not\in bA$.  Any
$\frakp\in\Ass(b)$ has $\frakp_{\frakp}$ principal by
Lemma~\ref{prime.principal}, hence height one, and $a/b\not\in
A_{\frakp}$ when $a\not\in bA_{\frakp}$.  Therefore to find an
associated prime $\frakp$ of $bA$ with $a\not\in bA_{\frakp}$
will complete the proof.  But this is easy since $a\in bA_{\frakp}$
for every $\frakp\in\Ass(b)$ if and only if $a=0$ in
$A_{\frakp}/bA_{\frakp}=(A/bA)_{\frakp}$ for every
$\frakp\in\Ass(b)$, which happens if and only if $a=0$ in $A/bA$ by
Lemma \ref{local.zero}, which happens if and only if $a\in bA$.
Since $a\neq 0$, such a prime must exist.
\end{proof}

\medskip
\section{Normalization}\label{sec:normalize}

If $A$ is a cancellative monoid, its normalization is the 
integral closure of $A$ in its group completion $A_0$. In contrast, 
consider the problem of defining the normalization of a 
non-cancellative monoid $A$, which should be something which has a 
kind of universal property for morphisms $A\to B$ with $B$ normal. 

We will restrict ourselves to the case when the
monoid $A$ is {\it partially cancellative}  (or {\it pc}), i.e., 
a quotient $A=C/I$ of a cancellative monoid $C$
(\cite[1.3, 1.20]{chww-p}). One advantage is that $A/\frakp$
is cancellative for every prime ideal $\frakp$ of a pc monoid,
and the normalization $(A/\frakp)_\nor$ of $A/\frakp$ exists. 

\begin{lem}\label{pcuniversalnormal}
If $A$ is a pc monoid and $f:A\to B$ is a morphism with $B$ normal,
then $f$ factors through the normalization of $A/\frakp$,
where $\frakp=\ker(f)$.
\end{lem}

\begin{proof}
The morphism $A/\frakp\to B$ of cancellative monoids induces a
homomorphism $f_0: (A/\frakp)_0\to B_0$ of their group completions. If 
$a\in(A/\frakp)_0$ belongs to $(A/\frakp)_\nor$ then there is an $n$
so that $a^n\in A/\frakp$. Then $b=f_0(a)\in B_0$
satisfies $b^n\in B$, so $b\in B$. Thus $f_0$ restricts to a map
$(A/\frakp)_\nor\to B$.
\end{proof}

\begin{subrem}\label{xz=yz}
The non-pc monoid $A=\langle x,y,z | xz=yz\rangle$ is
non-cancellative, reduced (\ref{nil(A)}) and even seminormal, yet
no map to a normal monoid has the universal property of a
normalization. Since $0$ is a prime ideal, there is no normalization
in the sense of Definition \ref{normaldefn} below, either.
Note that the disjoint union $X$ of $\MSpec(A/z)$ and $\MSpec(A/x\sim y)$
is normal but has the disadvantage that $X\to\MSpec(A)$ is not closed
(i.e., ``Going-up'' fails).
We have restricted to pc monoids in order to avoid these issues.
%
\end{subrem}

Thus the collection of maps $A\to(A/\frakp)_\nor$ has a kind
of universal property. However, a strict universal property is not 
possible within the category of monoids because monoids are local.
This is illustrated by the monoid $A=\langle x_1,x_2 | x_1x_2=0\rangle$;
see Example \ref{norm.axes} below.
Following the example of algebraic geometry, we will pass to 
the category of (pc) monoid schemes, where the normalization exists.

\begin{defn}\label{normaldefn}
Let $A$ be a pc monoid. The {\it normalization} $X_\nor$
of $X=\MSpec(A)$ is the disjoint union of the monoid schemes 
$\MSpec((A/\frakp)_\nor)$ as $\frakp$ runs over the 
minimal primes of $A$. By abuse of notation, we will refer to
$X_\nor$ as the normalization of $A$.

This notion is stable under localization: the normalization of 
$U=\MSpec(A[1/s])$ is an open subscheme of the normalization of 
$\MSpec(A)$; by Lemma \ref{local.normal}, its components are 
$\MSpec$ of the normalizations of the $(A/\frakp)[1/s]$ 
for those minimal primes $\frakp$ of A not containing $s$.

If $X$ is a pc monoid scheme, covered by affine opens $U_i$, one can
glue the normalizations $\widetilde{U}_i$ to obtain a normal monoid scheme
$X_\nor$, called the {\it normalization} of $X$. 
\end{defn}

\begin{rem}
The normalization $X_\nor$ is a normal monoid scheme: 
the stalks of $\cA_\nor$ are normal monoids.
It has the universal property that for every connected normal 
monoid scheme $Z$, every $Z\to X$ dominant on a component
factors uniquely through $X_\nor\to X$.
As this is exactly like \cite[Ex.\,II.3.8]{Hart}, we omit the details.
\end{rem}

Recall that the (categorical) product $A\times B$ of two pointed monoids
is the set-theoretic product with slotwise product and basepoint $(0,0)$.

\begin{lem}
Let $A$ be a pc monoid.
The monoid of global sections $H^0(X_\nor,\cA_\nor)$ of the 
normalization of $A$ is the product of the pointed monoids 
$(A/\frakp)_\nor$ as $\frakp$ runs over the minimal primes of $A$.
\end{lem}

\begin{proof}
For any sheaf $\cF$ on a disjoint union $X=\coprod X_i$, 
$H^0(X,\cF)=\prod H^0(X_i,\cF)$ by the sheaf axiom.
\end{proof}

\begin{exam}\label{norm.axes}
The normalization of $A=\langle x_1,x_2 | x_1x_2=0\rangle$
is the disjoint union of the affine lines $\langle x_i\rangle$.
The monoid of its global sections is
$\langle x_1\rangle\times\langle x_2\rangle$, and
is generated by $(1,0),(0,1),(x_1,1),(1,x_2)$.
\end{exam}

\goodbreak
{\it Seminormalization}

\medskip
Recall from \cite[1.7]{chww-p} that a reduced monoid $A$ is
{\it seminormal} if whenever $b,c\in A$ satisfy $b^3=c^2$
there is an $a\in A$ such that $a^2=b$ and $a^3=c$.
Any normal monoid is seminormal, and $\langle x,y|xy=0\rangle$
is seminormal but not normal. The passage from monoids to seminormal monoids 
(and monoid schemes) was critical in \cite{chww-p} for understanding
the behaviour of cyclic bar constructions under the 
resolution of singularities of a pc monoid scheme.

The {\it seminormalization} of a monoid $A$ is a seminormal monoid $A_\sn$,
together with an injective map $A_\red\to A_\sn$ such that 
every $b\in A_\sn$ has $b^n\in A_\red$ for all $n\gg0$. 
It is unique up to isomorphism, and any monoid map $A\to C$ with 
$C$ seminormal factors uniquely through $A_\sn$; see \cite[1.11]{chww-p}.
In particular, the seminormalization of $A$ lies between $A$ and
its normalization,  i.e., 
$\MSpec(A)_\nor\to\MSpec(A)$ factors through $\MSpec(A_\sn)$.

We shall restrict ourselves to the seminormalization of
pc monoids (and monoid schemes). By \cite[1.15]{chww-p}, if $A$ is a
pc monoid, the seminormalization of $A$  exists and is a pc monoid.
When $A$ is cancellative, $A_\sn$ is easy to construct.

\begin{exam}\label{sn.cancellative}
When $A$ is cancellative,
$A_\sn =\{ b\in A_0: b^n\in A\text{ for }n\gg0\}$;
this is a submonoid of $A_\nor$, and $A_\nor=(A_\sn)_\nor$. 
Since the normalization of a cancellative monoid induces a homeomorphism 
on the topological spaces $\MSpec$ \cite[1.6.1]{chww}, 
so does the seminormalization.
\end{exam}

If $A$ has more than one minimal prime, then 
$\MSpec(A)_\nor\to\MSpec(A)$ cannot be a bijection.
However, we do have the following result.

\begin{lem}\label{sn.homeo}
For every pc monoid $A$, $\MSpec(A_\sn)\to\MSpec(A)$ 
is a homeomorphism of the underlying topological spaces.
\end{lem}

\begin{proof}
Write $A=C/I$ for a cancellative monoid $C$, so $\MSpec(A)$ is the
closed subspace of $\MSpec(C)$ defined by $I$.
By \cite[1.14]{chww-p}, $A_\sn=C_\sn/(IC_\sn)$. Thus $\MSpec(A_\sn)$
is the closed subspace of $\MSpec(C_\sn)$ defined by $I$.
Since $\MSpec(C_\sn)\to\MSpec(C)$ is a homeomorphism 
(by \ref{sn.cancellative}), the result follows.
\end{proof}

The seminormalization of any pc monoid scheme exists
and has a universal property (see \cite[1.21]{chww-p}).
It may be constructed by glueing, since the seminormalization of 
$A$ commutes with localization \cite[1.13]{chww-p}.
Thus if $X$ is a pc monoid scheme then there are canonical maps
\[
X_\nor \to X_\sn\to X_\red\to X,
\]
and $X_\sn\to X$ is a homeomorphism by Lemma \ref{sn.homeo}.
We will return to this notion in Proposition \ref{Pic.Xsn}.

\bigskip
\section{Weil divisors}\label{sec:Weil}

Although the theory of Weil divisors is already interesting for
normal monoids, it is useful to state it for normal monoid schemes.

Let $X$ be a normal monoid scheme with generic monoid $A_0$.
Corollary \ref{DVmonoid} states that the stalk $\cA_x$ is a DV monoid
for every height one point $x$ of $X$. When $X$ is separated, a
discrete valuation on $A_0$ uniquely determines a point $x$ \cite[8.9]{chww}.

By a {\it Weil divisor} on $X$ we mean an element of the free abelian
group $\Div(X)$ generated by the height one points of $X$.
We define the divisor of $a\in A_0^\times$ to be the sum,
taken over all height one points of $A$:
\[
\div(a) = \sum_x v_x(a) x.
\]
When $A$ is of finite type, there are only finitely many prime ideals
in $A$, so this is a finite sum. Divisors of the form $\div(a)$
are called {\it principal divisors}. Since $v_x(ab)=v_x(a)+v_x(b)$,
the function $\div:A_0^\times\to\Div(X)$ is a group homomorphism, and
the principal divisors form a subgroup of $\Div(X)$.

\begin{defn}
The {\it Weil divisor class group} of $X$, written as $\Cl(X)$,
is the quotient of $\Div(X)$ by the subgroup of principal divisors.
\end{defn}

\begin{lem}\label{Cl.sequence}
If $X$ is a normal monoid scheme of finite type, there is an exact sequence
\[
1 \to \cA(X)^\times \to A_0^\times \map{\div} \Div(X) \to \Cl(X) \to 0.
\]
\end{lem}

\begin{proof}
We may suppose that $X$ is connected. It suffices to show that if
$a\in A_0^\times$ has $\div(a)=0$ then $a\in\cA(X)^\times$. This
follows from Theorem \ref{intersect.Ap}: 
when $X=\MSpec(A)$, $A$ is the intersection of the $A_x$.
\end{proof}

\begin{subex} 
(Cf.\,\cite[II.6.5.2]{Hart}) Let $A$ be the submonoid of $\Z^2_*$
generated by $x=(1,0)$, $y=(1,2)$ and $z=(1,1)$, and set $X=\MSpec(A)$.
(This is the toric monoid scheme $xy=z^2$.) Then $A$ has exactly two
prime ideals of height one: $p_1=(x,z)$ and $p_2=(y,z)$. Since 
$\div(x)=2p_1$ and $\div(z)=p_1+p_2$, we see that $\Cl(X)=\Z/2$.
\end{subex}

\begin{subex}
If $X$ is the non-separated monoid scheme obtained by gluing together 
$n+1$ copies of $\bA^1$ along the common (open) generic point, 
then $\Cl(X)=\Z^n$, as we see from Lemma \ref{Cl.sequence}. 
\end{subex}

If $U$ is an open subscheme of $X$, with complement $Z$, 
the standard argument \cite[II.6.5]{Hart} shows that there is a 
surjection $\Cl(X)\to\Cl(U)$, that it is an isomorphism if $Z$ 
has codimension $\ge2$, and that if $Z$ is the closure of a height one
point $z$ then there is an exact sequence
\[
\Z\; \map{z} \Cl(X) \to \Cl(U) \to 0.
\]

\begin{prop}
$\Cl(X_1\times X_2) = \Cl(X_1)\oplus\Cl(X_2).$
\end{prop}

\begin{proof}
By \cite[3.1]{chww}, the product monoid scheme exists, and its underlying
topological space is the product. Thus a codimension one point of
$X_1\times X_2$ is either of the form $x_1\times X_2$ or $X_1\times x_2$.
Hence $\Div(X_1\times X_2) \cong \Div(X_1)\oplus\Div(X_2)$.
It follows from Lemma \ref{wedge.normal} that $X_1\times X_2$ 
is normal, and the pointed monoid at its generic point is the
smash product of the pointed monoids $A_1$ and $A_2$ of $X_1$ and $X_2$ at
their generic points. If $a_i\in A_i$ then the principal divisor of
$a_1\wedge a_2$ is $\div(a_1)+\div(a_2)$. Thus
\[
\Cl(X_1\times X_2)=
\frac{\Div(X_1)\oplus\Div(X_2)}{\div(A_1)\oplus\div(A_2)}
\cong \Cl(X_1)\oplus\Cl(X_2).   \qedhere
\]
\end{proof}

\begin{exam}\label{normal.v.toric}
By \cite[4.5]{chww}, any connected separated normal monoid scheme $X$
decomposes as the product of a toric monoid scheme $X_\Delta$ and
$\MSpec(U_*)$ for some finite abelian group $U$. 
($U$ is the group of units of $X$.) Since 
$U_*$ has no height one primes, $\Div(X)=\Div(X_\Delta)$ and 
the Weil class group of $X$ is $\Cl(X_\Delta)$,
the Weil class group of the associated toric monoid scheme.

By construction \cite[4.2]{chww}, the points of $X_\Delta$ 
correspond to the cones of the fan $\Delta$ and the height one points 
of $X_\Delta$ correspond to the edges in the fan. Thus our Weil
divisors correspond naturally to what Fulton calls a ``$T$-Weil
divisor'' on the associated toric variety $X_k$ (over a field $k$)
in \cite[3.3]{F}.
Since the group completion $A_0$ is the free abelian group $M$
associated to $\Delta$, it follows from \cite[3.4]{F} that our
Weil divisor class group $\Cl(X_\Delta)$ is isomorphic to the
Weil divisor class group $\Cl(X_k)$ of associated toric variety.
\end{exam}

\newpage
\section{Invertible sheaves}\label{sec:Pic}

Let $X$ be a monoid scheme with structure sheaf $\cA$. An
{\it invertible sheaf} on $X$ is a sheaf $\cL$ of $\cA$-sets
which is locally isomorphic to $\cA$ in the Zariski topology.
If $\cL_1,\cL_2$ are invertible sheaves, their smash product is the
sheafification of the presheaf $U\mapsto \cL_1(U)\wedge_{\cA(U)}\cL_2(U)$;
it is again an invertible sheaf. Similarly, $\cL^{-1}$ is the
sheafification of $U\mapsto \Hom_{\cA}(\cL(U),\cA(U))$, and 
evaluation $\cL\wedge_{\cA}\cL^{-1}\map{\sim}\cA$ is an isomorphism.
Thus the set of isomorphism classes of invertible sheaves on $X$ is a 
group under the smash product.

\begin{defn}\label{def:Pic}
The {\it Picard group} $\Pic(X)$ is the group of isomorphism classes
of invertible sheaves on $X$.
\end{defn}

Since a monoid $A$ has a unique maximal ideal (the non-units),
an invertible sheaf on $\MSpec(A)$ is just an $A$-set isomorphic to $A$.
This proves:

\begin{lem}\label{Pic.affine}
For every affine monoid scheme $X=\MSpec(A)$, $\Pic(X)=0$.
\end{lem}

For any monoid $A$, the group of $A$-set automorphisms of $A$ is
canonically isomorphic to $A^\times$. Since the subsheaf $\Gamma$ of
generators of an invertible sheaf $\cL$ is a torsor for $\cA^\times$,
and $\cL=\cA\wedge_{\cA^\times}\Gamma$, this proves:

\begin{lem}\label{Pic=H1}
$\Pic(X) \cong H^1(X,\cA^\times)$.
\end{lem}

Recall that a morphism $f:Y\to X$ of monoid schemes is {\it affine}
if $f^{-1}(U)$ is affine for every affine open $U$ in $X$;
see \cite[6.2]{chww}.

\begin{prop}\label{affine.f_*}
If $f:Y\to X$ is an affine morphism of monoid schemes, then the
direct image $f_*$ is an exact functor from sheaves (of abelian groups) 
on $Y$ to sheaves on $X$. In particular, 
$H^*(Y,\cL)\cong H^*(X,p_*\cL)$ for every sheaf $\cL$ on $Y$.
\end{prop}

\begin{proof}
Suppose that $0\to\cL'\to\cL\to\cL''\to0$ is an exact sequence
of sheaves on $Y$. Fix an affine open  $U=\MSpec(A)$ of $X$ with 
closed point $x\in X$. Then $f^{-1}(U)=\MSpec(B)$ for some
monoid $B$. If $y\in Y$ is the unique closed point of $\MSpec(B)$
the stalk sequence $0\to\cL'_y\to\cL_y\to\cL''_y\to0$ is exact.
Since this is the stalk sequence at $x$ of 
$0\to f_*\cL'\to f_*\cL\to f_*\cL''\to0$, 
the direct image sequence is exact.
\end{proof}

Here is an application, showing one way in which monoid schemes 
differ from schemes.
Let $T$ denote the free (pointed) monoid on generator $t$, and
let $\bA^1$ denote $\MSpec(T)$. Then $A\wedge T$ is the analogue of 
a polynomial ring over $A$, and $X\times\bA^1$ is the monoid scheme
which is locally $\MSpec(A)\times\bA^1=\MSpec(A\wedge T)$; 
see \cite[3.1]{chww}. Thus $p:X\times\bA^1\to X$ is affine,
and $f_*\cA_Y^\times=\cA_X^\times$.
From Proposition \ref{affine.f_*} we deduce

\begin{cor}\label{Pic.HI}
For every monoid scheme $X$,
$\Pic(X)\cong\Pic(X\times\bA^1)$.
\end{cor}

\newpage
\section{Cartier divisors}\label{sec:Cartier}

Let $(X,\cA)$ be a cancellative monoid scheme. We write $A_0$ for the
stalk of $\cA$ at the generic point of $X$, and 
$\cA_0$ for the associated constant sheaf.
A {\it Cartier divisor} on $X$ is a global section of the sheaf of groups
$\cA_0^\times/\cA^\times$. On each affine open $U$, it is given by an 
$a_U\in A_0^\times$ up to a unit in $\cA(U)^\times$, and we have the 
usual representation as $\{(U,a_U)\}$ with $a_U/a_V$ in 
$\cA(U\cap V)^\times$. We write $\Cart(X)$ for the group of Cartier 
divisors on $X$.
The {\it principal} Cartier divisors, i.e., those represented by 
some $a\in A_0^\times$, form a subgroup of $\Cart(X)$.

\begin{prop}\label{Cart.Pic}
Let $X$ be a cancellative monoid scheme. Then the map $D\mapsto\cL(D)$
defines an isomorphism between the group of Cartier divisors
modulo principal divisors and $\Pic(X)$.
\end{prop}

\begin{proof}
Consider the short exact sequence of sheaves of abelian groups
\[
1 \to\cA^\times \to \cA_0^\times \to \cA_0^\times/\cA^\times \to 1.
\]
Since $\cA_0^\times$ is constant and $X$ is irreducible we have 
$H^1(X,\cA_0^\times)=0$ \cite[III.2.5]{Hart}. By Lemma \ref{Pic=H1},
the cohomology sequence becomes:
\[
0\to\cA(X)^\times \to \cA_0^\times \map{\div} \Cart(X) 
\map{\delta}\Pic(X)\to 0. \qedhere
\]
\end{proof}

\begin{exam}\label{ex:L(D)}
If $D$ is a Cartier divisor on a cancellative monoid scheme $X$,
represented by $\{(U,a_U)\}$, we define a subsheaf $\cL(D)$ of the
constant sheaf $\cA_0$ by letting its restriction to $U$ be generated
by $a_U^{-1}$. This is well defined because $a_U^{-1}$ and $a_V^{-1}$
generate the same subsheaf on $U\cap V$. The usual argument
\cite[II.6.13]{Hart} shows that $D\mapsto\cL(D)$ defines an isomorphism
from $\Cart(X)$ to the group of invertible subsheaves of $A_0^\times$.
By inspection, the map $\delta$ in \ref{Cart.Pic} sends $D$ to $\cL(D)$.
\end{exam}

\begin{lem}
If $X$ is a normal monoid scheme of finite type, 
$\Pic(X)$ is a subgroup of $\Cl(X)$.
\end{lem}

\begin{proof}
Every Cartier divisor $D=\{(U,a_U)\}$ determines a Weil divisor; 
the restriction of $D$ to $U$ is the divisor of $a_U$. 
It is easy to see that this makes the Cartier divisors into a 
subgroup of the Weil divisor class group $D(X)$, under which 
principal Cartier divisors are identified with principal 
Weil divisors. This proves the result.
\end{proof}

\begin{thm}
Let $X$ be a separated connected monoid scheme.
If $X$ is locally factorial then every Weil divisor is a Cartier
divisor, and $\Pic(X)=\Cl(X)$.
\end{thm}

\begin{proof}
By Example \ref{ex:normal}, $X$ is normal since factorial monoids are normal.
Thus $\Pic(X)$ is a subgroup of $\Cl(X)$, and it suffices to show
that every Weil divisor $D=\sum n_i x_i$ is a Cartier divisor.
For each affine open $U$, and each point $x_i$ in $U$,
let $p_i$ be the generator of the prime ideals associated to $x_i$; 
then the divisor of $a_U=\prod p_i^{n_i}$ is the restriction of $D$ 
to $U$, and $D=\{(U,a_U)\}$.
\end{proof}

\begin{lem}\label{Pic.Pn}
For the projective space monoid scheme $\P^n$ we have 
$$\Pic(\P^n)=\Cl(\P^n)=\Z.$$
\end{lem}

\begin{subrem}
This calculation of $\Pic(\P^n)$ formed the starting point of our 
investigation. We learned it from Vezzani (personal communication), 
but it is also found in \cite{CLS} and \cite{GHS}. 
Related calculations are in \cite{Hutt} and \cite{Sz}.
\end{subrem}

\begin{proof}
Since $\P^n$ is locally factorial, $\Pic(\P^n)=\Cl(\P^n)$.
By definition, $\P^n$ is $\MProj$ of the free abelian monoid on
$\{ x_0,...,x_n\}$, and $A_0$ is the free abelian group with the
$x_i/x_0$ as basis ($i=1,...,n$). On the other hand, $\Div(\P^n)$
is the free abelian group on the generic points $[x_i]$ of the $V(x_i)$.
Since $\div(x_i/x_0)=[x_i]-[x_0]$, the result follows.
\end{proof}

Let $\Delta$ be a fan, $X$ the toric monoid scheme associated to
$\Delta$ by \cite[4.2]{chww}, and $X_k$ the usual toric variety 
associated to $\Delta$ over some field $k$. ($X_k$ is the 
$k$-realization $X_k$ of $X$.) As pointed out in 
Example \ref{normal.v.toric}, our Weil divisors correspond to the 
$T$-Weil divisors of the toric variety $X_k$ and $\Cl(X)\cong\Cl(X_k)$.
Moreover, our Cartier divisors on $X$ correspond to the
$T$-Cartier divisors of \cite[3.3]{F}).
Given this dictionary, the following result is established by Fulton 
in \cite[3.4]{F}.

\begin{thm}\label{thm:toric}
Let $X$ and $X_k$ denote the toric monoid scheme and toric variety (over $k$)
associated to a given fan. Then
$\Pic(X) \cong \Pic(X_k).$

Moreover, $\Pic(X)$ is free abelian if
$\Delta$ contains a cone of maximal dimension.
\end{thm}

\section{$\Pic$ of pc monoid schemes}

In this section, we derive some results about the Picard group of
pc monoid schemes.
When $X$ is a pc monoid scheme, we can form the reduced monoid scheme 
$X_\red=(X,\cA_\red)$
using Remark \ref{nil(A)}: the stalk of $\cA_\red$ at $x$ 
is $\cA_x/\nil(\cA_x)$.  Since $\cA^\times=\cA_\red^\times$, 
the map $X_\red\to X$ induces an isomorphism $\Pic(X)\cong\Pic(X_\red)$.

We will use the constructions of normalization and seminormalization
given in Section \ref{sec:normalize}.

\begin{prop}\label{Pic.Xsn}
If $X$ is a pc monoid scheme, the canonical map $X_\sn\to X$
induces an isomorphism $\Pic(X)\cong\Pic(X_\sn)$.
\end{prop}

\begin{proof}
Since $X_\red$ and $X$ have the same underlying space, it suffices 
by Lemma \ref{Pic=H1} to assume that $X$ is reduced and show that 
the inclusion $\cA^\times\to\cA_\sn^\times$ is an isomorphism.
It suffices to work stalkwise, so we are reduced to showing that
if $A$ is reduced then $A^\times\to A_\sn^\times$ is an isomorphism.
If $b\in A_\sn^\times$ then both $b^n$ and $(1/b)^n$ are in $A$
for large $n$, and hence both $b=b^{n+1}b^{-n}$ and 
$b^{-1}=b^n(1/b)^{1+n}$ are in $A$, so $b\in A^\times$.
\end{proof}

\begin{lem}\label{Pic.Xnor}
Let $X$ be a cancellative seminormal monoid scheme and
$p:X_\nor\to X$ its normalization. If $\cH$ denotes the sheaf
$p_*(\cA_\nor^\times)/\cA^\times$ on $X$,
there is an exact sequence
\[
1\to\cA(X)^\times\to\cA_\nor(X_\nor)^\times \to H^0(X,\cH) \to
\Pic(X) \map{p^*} \Pic(X_\nor) \to H^1(X,\cH).
\]
\end{lem}

\begin{proof}
At each point $x\in X$, the stalk $A=\cA_x$ is a submonoid of its 
normalization $A_\nor=p_*(\cA_\nor)_x$ (by Lemma \ref{local.normal})
and we have an exact sequence of sheaves on $X$:
\[ 
1 \to \cA^\times \to p_*(A_\nor^\times) \to \cH\to 1. 
\]
Since $p$ is affine, Proposition \ref{affine.f_*} implies that
$\cA_\nor(X_\nor)^\times=H^0(X,p_*\cA_\nor^\times)$ and
$\Pic(X_\nor)=H^1(X,p_*A_\nor^\times)$, and the associated
cohomology sequence is the displayed sequence.
\end{proof}

Here are two examples showing that $\Pic(X)\to\Pic(X_\nor)$ need not
be an isomorphism when $X$ is seminormal and cancellative. 

\begin{exam}
Let $A_+$ (resp., $A_-$) be the submonoid of the free monoid
$B=\langle x,y\rangle$ generated by $\{ x,y^2,xy\}$ 
(resp., $\{ x,y^{-2},xy^{-1}\}$). These are seminormal but not normal.
If $X$ is the monoid scheme obtained by gluing the $U_\pm=\MSpec(A_\pm)$
together along $\MSpec(\langle x,y^2,y^{-2}\rangle)$ then it is easy
to see that $\Pic(X)=\Z$, with a generator represented by 
$(U_+,y^2)$ and $(U_-,1)$.  The normalization $X_\nor$ is the toric
monoid scheme $\bA^1\times\P^1$, and $\Pic(X_\nor)\cong\Z$,
with a generator represented by $(U_+,y)$ and $(U_-,1)$. 
Thus $\Pic(X)\to\Pic(X_\nor)$ is an injection with cokernel $\Z/2$.
\end{exam}

\begin{exam}
Let $U$ be an abelian group and $A_x$ the submonoid of 
$B=U_*\wedge\langle x\rangle$ consisting of $0,1$ and all terms
$ux^n$ with $u\in U$ and $n>0$. Then $A_x$ is seminormal and $B$ is its
normalization. Let $X$ be obtained by gluing $\MSpec(A_x)$ and
$\MSpec(A_{1/x})$ together along their common generic point,
$\MSpec(U_*\wedge\langle x,1/x\rangle)$. The normalization of $X$
is $X_\nor=\MSpec(U_*)\times\P^1$, and $\Pic(X_\nor)=\Z$ by
Example \ref{normal.v.toric} and Lemma \ref{Pic.Pn}.
Because $p_*(\cA_\nor^\times)/\cA^\times$ is a skyscraper sheaf 
with stalk $U$ at the two closed points, we see from Lemma \ref{Pic.Xnor}
that $\Pic(X)=\Z\times U$. Thus $\Pic(X)\to\Pic(X_\nor)$
is a surjection with kernel $U$.
\end{exam}

Finally, we consider the case when $X$ is reduced pc monoid scheme
which is not cancellative. We may suppose that $X$ is of finite type,
so that the stalk at a closed point is an affine open $\MSpec(A)$
with minimal points $\frakp_1,...,\frakp_r$, $r>1$.
Then the closure $X'$ of $\frakp_1$ is a cancellative seminormal
monoid scheme. Let $X''$ denote the closure of the remaining minimal
points of $X$, and set $X'''=X'\cap X''$.
Then we have the exact Mayer-Vietoris sequence of sheaves on $X$:
\[
0 \to \cA_X \to p'_*\cA_{X'}\times p''_*\cA_{X''}
\to p'''_*\cA_{X'''} \to 1.
\]
Because the immersions are affine, Proposition \ref{affine.f_*}
yields the exact sequence
\begin{equation}\label{eq:MV}
1\to\cA(X)^\times \to \cA_{X'}^\times \times \cA_{X''}^\times
\to \cA_{X'''}^\times \to 
\Pic(X)\to \Pic(X')\times\Pic(X'')\to \Pic(X''').
\end{equation}
The Picard group may then be determined by induction on $r$ and $\dim(X)$.

\begin{exam}
If $X$ is obtained by gluing together $X_1,..., X_n$ at a common
generic point, then \eqref{eq:MV} yields $\Pic(X)=\oplus\Pic(X_i)$.
\end{exam}


\end{document}